\documentclass[openany,reqno,a4paper,11pt]{amsart}

\usepackage{amsmath,amssymb,amsthm}

\usepackage{mathrsfs}

\usepackage{mathtools}
\usepackage{commath}

\usepackage{thmtools}
\usepackage{thm-restate}

\usepackage{cases}
\usepackage{enumitem}
\setlist[enumerate,1]{label=(\roman*)}

\usepackage[pdftex, pdfborderstyle={/S/U/W 0}]{hyperref}
\hypersetup{
    colorlinks=true,
    linkcolor=magenta,
    citecolor=cyan,
}

\usepackage{cleveref}

\usepackage{etoolbox}

\usepackage{comment}

\usepackage[numbers]{natbib}

\numberwithin{equation}{section}

\ifdefined\thmcolor
\declaretheoremstyle[
  shaded={bgcolor=\thmcolor}
]{plain}
\else
\fi

\ifdefined\defcolor
\declaretheoremstyle[
  headfont=\normalfont\bfseries,
  bodyfont=\normalfont,
  shaded={bgcolor=\defcolor}
]{noital}
\else
\declaretheoremstyle[
  headfont=\normalfont\bfseries,
  bodyfont=\normalfont,
]{noital}
\fi

\declaretheorem[style=plain,numberwithin=section,name=Theorem]{theorem}
\declaretheorem[style=plain,sibling=theorem,name=Proposition]{proposition}
\declaretheorem[style=plain,sibling=theorem,name=Lemma]{lemma}

\declaretheorem[style=plain,sibling=theorem,name=Conjecture]{conjecture}
\declaretheorem[style=plain,sibling=theorem,name=Claim]{claim}

\declaretheorem[style=plain,sibling=theorem,name=Question]{question}

\declaretheorem[style=plain,numbered=no,name=Theorem]{theorem-n}
\declaretheorem[style=plain,numbered=no,name=Proposition]{proposition-n}
\declaretheorem[style=plain,numbered=no,name=Lemma]{lemma-n}
\declaretheorem[style=plain,numbered=no,name=Corollary]{corollary-n}
\declaretheorem[style=plain,numbered=no,name=Conjecture]{conjecture-n}
\declaretheorem[style=plain,numbered=no,name=Claim]{claim-n}
\declaretheorem[style=plain,numbered=no,name=Fact]{fact-n}
\declaretheorem[style=plain,numbered=no,name=Open Problem]{openproblem-n}
\declaretheorem[style=plain,numbered=no,name=Question]{question-n}

\declaretheorem[style=noital,sibling=theorem,name=Remark]{remark}
\declaretheorem[style=noital,sibling=theorem,name=Definition]{definition}

\declaretheorem[style=noital,numbered=no,name=Remark]{remark-n}
\declaretheorem[style=noital,numbered=no,name=Definition]{definition-n}
\declaretheorem[style=noital,numbered=no,name=Construction]{construction-n}
\declaretheorem[style=noital,numbered=no,name=Example]{example-n}

\newcommand{\defined}{\mathrel{\coloneqq}}

\newcommand{\st}{\mathbin{\colon}}

\undef{\set}
\DeclarePairedDelimiter{\set}{\lbrace}{\rbrace}

\undef{\emptyset}
\newcommand{\emptyset}{\varnothing}

\newcommand{\union}{\mathbin{\cup}}
\newcommand{\inter}{\mathbin{\cap}}

\undef{\abs}
\DeclarePairedDelimiterX{\abs}[1]
  {\lvert}{\rvert}{\ifblank{#1}{\,\cdot\,}{#1}}

\undef{\norm}
\DeclarePairedDelimiterX{\norm}[1]
  {\lVert}{\rVert}{\ifblank{#1}{\,\cdot\,}{#1}}

\DeclarePairedDelimiterX{\inner}[2]
  {\langle}{\rangle}{\ifblank{#1}{\,\cdot\,}{#1},\ifblank{#2}{\,\cdot\,}{#2}}

\DeclareMathDelimiter{\given}
  {\mathbin}{symbols}{"6A}{largesymbols}{"0C}

\DeclareMathOperator{\Prob}{\mathbb{P}}
\DeclarePairedDelimiterXPP{\prob}[1]
  {\Prob}{\lparen}{\rparen}{}
  {\renewcommand{\given}{\nonscript\;\delimsize\vert\nonscript\;\mathopen{}}#1}

\DeclareMathOperator{\Expec}{\mathbb{E}}
\DeclarePairedDelimiterXPP{\expec}[1]
  {\Expec}{\lparen}{\rparen}{}
  {\renewcommand{\given}{\nonscript\;\delimsize\vert\nonscript\;\mathopen{}}#1}

\DeclareMathOperator{\Var}{Var}
\DeclarePairedDelimiterXPP{\var}[1]
  {\Var}{\lparen}{\rparen}{}
  {\renewcommand{\given}{\nonscript\;\delimsize\vert\nonscript\;\mathopen{}}#1}

\DeclareMathOperator{\Cov}{Cov}
\DeclarePairedDelimiterXPP{\cov}[2]
  {\Cov}{\lparen}{\rparen}{}{#1,#2}

\newcommand{\sseq}{\subseteq}

\newcommand{\cE}{\mathcal{E}}

\newcommand{\cH}{\mathcal{H}}

\newcommand{\cP}{\mathcal{P}}

\usepackage{geometry}
\usepackage{booktabs}

\usepackage{tikz}
\usetikzlibrary {quotes} 

\newcommand{\conn}[1]{\xrightarrow{#1}}
\newcommand{\biconn}[1]{\xleftrightarrow{#1}}

\newcommand{\bbg}{\Tilde{G}}
\newcommand{\bbq}{\Tilde{Q}}
\newcommand{\bbd}{\Tilde{D}}
\newcommand{\bbh}{\Tilde{\cH}}
\newcommand{\conj}[1]{\text{BBC}(#1)}

\newcommand{\mone}{E_2^{T_2}(G_2)}
\newcommand{\mtwo}{E_1(G_1)}
\newcommand{\mthree}{E_4^{T_4}(\cH_4)}
\newcommand{\mfour}{E_5(\cH_5)}
\newcommand{\mfive}{E_6^{T_6,F_6}(D_6)}
\newcommand{\msix}{E_7(D_7)}

\newcommand{\czero}{\conj{E_0^T,G}}
\newcommand{\cone}{\conj{E_2^{T_2},G_2}}
\newcommand{\ctwo}{\conj{E_1,G_1}}
\newcommand{\cthree}{\conj{E_4^{T_4},\cH_4}}
\newcommand{\cfour}{\conj{E_5,\cH_5}}
\newcommand{\cfive}{\conj{E_6^{T_6,F},D_6}}
\newcommand{\csix}{\conj{E_7,D_7}}

\newcommand{\connone}{\biconn{\mone}}
\newcommand{\conntwo}{\biconn{\mtwo}}
\newcommand{\connthree}{\biconn{\mthree}}
\newcommand{\connfour}{\biconn{\mfour}}
\newcommand{\connfive}{\conn{\mfive}}
\newcommand{\connsix}{\conn{\msix}}

\declaretheorem[style=plain,sibling=theorem,name=Model]{model}

\begin{document}

\title{The bunkbed conjecture is not robust to generalisation}

\author{Lawrence Hollom}
\address{Department of Pure Mathematics and Mathematical Statistics (DPMMS), University of Cambridge, Wilberforce Road, Cambridge, CB3 0WA, United Kingdom}
\email{lh569@cam.ac.uk}



\begin{abstract}
  The bunkbed conjecture, which has featured in the folklore of probability theory since at least 1985, concerns bond percolation on the product graph $G\Box K_2$. 
  We have two copies $G_0$ and $G_1$ of $G$, and if $x^{(0)}$ and $x^{(1)}$ are the copies of a vertex $x\in V(G)$ in $G_0$ and $G_1$ respectively, then edge $x^{(0)}x^{(1)}$ is present.
  The conjecture states that, for vertices $u,v\in V(G)$, percolation from $u^{(0)}$ to $v^{(0)}$ is at least as likely as percolation from $u^{(0)}$ to $v^{(1)}$.
  While the conjecture is widely expected to be true, having attracted significant attention, a general proof has not been forthcoming.
  
  In this paper we consider three natural generalisations of the bunkbed conjecture; to site percolation, to hypergraphs, and to directed graphs.
  Our main aim is to show that all these generalisations are false, and to this end we construct a sequence of counterexamples to these statements.
  However, we also consider under what extra conditions these generalisations might hold, and give some classes of graph for which the bunkbed conjecture for site percolation does hold.
\end{abstract}

\maketitle


\section{Introduction}
\label{sec:intro}

In 1985, Kasteleyn posed the bunkbed conjecture, which then circulated as an intriguing problem in percolation theory (see \cite[Remark 5]{KahnBerg2001}).
Interest in the problem has waxed and waned over time, but, especially in recent years, the conjecture has attracted significant attention.
A feature of the problem which makes it particularly appealing is its apparent simplicity and how, as has been noted by several authors, it appears to be intuitively clear.

We now provide some definitions and state our main results before returning to the history of the problem.
While most work on the topic is concerned with simple graphs, we work in a more general setting, and thus define the bunkbed graph in several cases, as follows.

\begin{definition}[Bunkbed graph]
    For $Q$ a graph, digraph, or hypergraph, we can define the bunkbed structure $\bbq$ as follows.
    Take two copies, $Q_0$ and $Q_1$, of $Q$, labelling the copy of vertex $v$ and edge $e$ in $Q_i$ as $v^{(i)}$ and $e^{(i)}$ respectively.
    We will call such edges $e^{(i)}$ \emph{horizontal edges}.
    $E(\bbq)$ then consists of $E(Q_0)\union E(Q_1)$ along with an edge between $v^{(0)}$ and $v^{(1)}$ for each $v\in V(Q)$, which will be called a \emph{vertical edge}.
    In the case of directed graphs, this vertical edge will be bidirectional.
    $Q_0$ will be called the \emph{lower bunk} and $Q_1$ the \emph{upper bunk} of $\bbq$.
\end{definition}

We will be concerned in this paper with percolation on the bunkbed graph under various uniform models of percolation.
We will consider two types of model: site percolation and bond percolation.
In both cases, we will ultimately take a family $\cE(\bbq)$ of subsets of $E(\bbq)$, which we will sample from uniformly at random to produce some set $X\sseq E(\bbq)$ of edges which will be retained.
We now define these concepts formally.

\begin{definition}[Percolation model]
    If $Q$ is a graph, digraph, or hypergraph, then a \emph{percolation model} is a family $\cE(\bbq)$ is a family of subsets of $E(\bbq)$.
    We perform percolation by selecting an element $X\in \cE(\bbq)$ uniformly at random, and then retain precisely those edges in $X$.
    For two vertices $x,y\in V(\bbq)$, we write
    \begin{align*}
        \Prob(x\conn{\cE(\bbq)}y) \qquad \text{ or } \qquad \Prob(x\biconn{\cE(\bbq)}y)
    \end{align*}
    for the probability that there is a directed path of retained edges from $x$ to $y$ in the former, or a bidirectional path in the latter case, in $\bbq$ under model $\cE(\bbq)$.
    When $Q$ is a hypergraph, a path in $Q$ is a sequence of vertices in which consecutive vertices are both members of some hyperedge in $Q$.
\end{definition}

We can then use this definition to define site percolation as a special case of bond percolation.
We will only consider site percolation on graphs, so for simplicity we only define this case.

\begin{definition}[Site percolation model]
    If $G$ is a graph, then a percolation model $\cE(\bbg)$ is said to be a \emph{site percolation model} if, for every $X\in\cE(\bbg)$ there is some set $U\sseq V(\bbg)$, which we shall refer to as the set of \emph{open vertices}, such that $X$ consists of precisely those edges with both ends in $U$.
\end{definition}

We consider site percolation as a special case of our general (bond) percolation model for convenience; it will be useful for us to still have all vertices present, even if they are in no edges.

We remark that in much other work on the topic, percolation is phrased in terms of sampling edges independently at random with some probability $p$.
Our sampling as above is equivalent to this in the case that $\cE(\bbq)$ is the powerset $\cP(E(\bbq))$, and $p=1/2$.
This restriction to $p=1/2$ turns out to not be important, as will be discussed later.

We now provide a very general statement of the bunkbed conjecture in terms of a general model, before we proceed to consider several special cases.

\begin{definition}[Bunkbed conjecture for a general model $\cE$]
\label{conj:bunkbed-general}
    Let $Q$ be a graph, hypergraph, or digraph, $\cE(\bbq)$ be a model of percolation on $\bbq$ as defined above.
    \emph{The bunkbed conjecture for model $\cE$} states that for any $Q$, and any vertices $u,v\in V(Q)$, we have
    $$\Prob(u^{(0)}\conn{\cE(\bbq)}v^{(0)})\geq \Prob(u^{(0)}\conn{\cE(\bbq)}v^{(1)}).$$
    Let $\conj{\cE,Q}$ denote the above statement.
\end{definition}

There are several different statements which are referred to as `the bunkbed conjecture', and they are used somewhat interchangeably in the literature; Rudzinski and Smyth \cite{Rudzinski2016} detail a few of these statements. 
One such form of the conjecture, which fits the notation we have described above, concerns the following model of percolation.

\begin{model}[$E_0^T$, the standard bunkbed model]
\label{model:standard}
    For an undirected graph $G$ and a fixed set $T\sseq V(G)$, a set $X\sseq E(\bbg)$ of horizontal edges is chosen uniformly at random, and retained along with precisely those vertical edges $v^{(0)}v^{(1)}$ such that $v\in T$. 
    Equivalently, each horizontal edge of $\bbg$ is retained independently at random with probability $0.5$.
\end{model}

In this setting, $\czero$ is equivalent to many statements referred to as `the bunkbed conjecture'.
When stated, the conjecture is often given in a slightly more general form, retaining each horizontal edge of $\bbg$ with probability $p$ for some fixed $p\in(0,1)$, as in for example the work of H\"{a}ggstr\"{o}m \cite{haggstrom1998conjecture,haggstrom2003probability}.
However, as proved by Rudzinski and Smyth \cite[Corollary 1.7]{Rudzinski2016}, the case for any specific probability $p$ implies the general case, and so the truth of $\czero$ would suffice to prove the more general conjecture.

The conjecture $\conj{\cE,G}$ is false for general percolation models $\cE$; this is unsurprising as our definition has no requirement of symmetry, and so permits many pathological models.
However, we will show that $\conj{\cE,G}$ is also false for several highly symmetric models $\cE$, some of which have been conjectured to be true by various authors.
This is particularly surprising as much of the same intuition which tells us that $\czero$ should be true also applies to these other -- false -- cases.

The conjecture has attracted significant attention over the last few years, being resolved in progressively more general special cases.
One initial sequence of results by de Buyer and then van Hintum and Lammers \cite{Buyer2016, Buyer2018, VanHintum2019} resolved the case of the complete graph.
This was then further generalised by Richthammer \cite{Richthammer2022} to a broader class of highly symmetric graphs.
Linusson \cite{linusson2011,linusson2019} took a different route, and both considered generalisations of the conjecture and related the conjecture to a problem concerning randomly directed graphs.
This work was then built upon by Hutchcroft, Nizi\'{c}-Nikolac, and Kent \cite{hutchcroft2023bunkbed}, who resolved the conjecture in the case of general $G$ and probability $p\uparrow 1$.
A different angle on the $p\uparrow 1$ limit was later considered by the author \cite{hollom2024new}, proving the conjecture holds in this case by working with graph cuts.

We now turn our attention to more general models; it is natural in an attempt to prove the bunkbed conjecture to try to generalise the statement of the problem, in the hope that this may be conducive to an inductive proof.
Indeed, upon attacking the conjecture $\conj{E_0,G}$, a generalisation to hypergraphs soon arises as a tempting possibility.
Likewise, directed graphs allow for the possibility of attaching `gadgets' to a graph to easily generalise the conjecture further, and on generalising to site percolation, one no longer needs to consider paths which return to an already-visited pair of vertices in the bunkbed graph.
However, our main result is the following, which says that all of these generalisations are in fact false.

\begin{theorem}
\label{thm:main}
    The generalised bunkbed conjecture is false for site percolation, hypergraphs, and digraphs.
    More precisely, there are counterexamples to the statement $\conj{\cE,Q}$ for each of the models $E_1,E_2^T,E_4^T,E_5,E_6^{T,F},E_7$ to be discussed in \Cref{sec:site-percolation,sec:models}.
\end{theorem}

In view of this general negative result, it is still interesting to consider for which graphs these variations of the bunkbed conjecture might be true.
To this end, we will also prove the following positive result in the case of site percolation and some highly structured graphs $G$.

We recall that a \emph{wheel} graph consists of a cycle $C$ plus an additional vertex connected by an edge to every vertex of $C$.

\begin{theorem}
\label{thm:simple}
    Let $G$ be a path, a cycle, or a wheel, and let $u,v\in V(G)$ be vertices. 
    Let $S\sseq V(\bbg)$ be selected uniformly at random amongst all such subsets.
    It is at least as likely that there is a path entirely in $S$ from $u^{(0)}$ to $v^{(0)}$ as it is that there is a path in $S$ from $u^{(0)}$ to $v^{(1)}$.
\end{theorem}

We now provide an outline of the rest of the paper.


\subsection{Paper overview}

Before diving into our results, we first describe in \Cref{sec:prelims} the terminology we will use, and also state and prove a simple lemma concerning a symmetry that we will often make use of.

In \Cref{sec:site-percolation} we then focus on the case of site percolation, and the main aim of this section is to construct a counterexample to the bunkbed conjecture for site percolation.
In pursuit of this goal, we will need to work with a more complex model of site percolation first, and then translate our counterexample from this setting to the general case.
We then, in \Cref{subsec:positive}, prove \Cref{thm:simple}, a positive result.

After dealing with site percolation, we then move on to directed graphs and hypergraphs.
In \Cref{sec:models}, we introduce the models of percolation that we will deal with, and then in \Cref{sec:constructions} we build the hypergraphs and digraphs which will serve as our counterexamples to various generalisations of the bunkbed conjecture.
Once this section is completed, we will have proved \Cref{thm:main}.

Finally, in \Cref{sec:conclusion}, we discuss potential extensions of our work, and state several open problems.


\section{Terminology and preliminary results}
\label{sec:prelims}

We now introduce some terminology which we shall use throughout the paper, and then a lemma which we will use on several occasions.
We now define several useful notions.
Firstly, we restrict our attention to models with a particular symmetry that is convenient to work with.

\begin{definition}[Symmetric percolation models]
    For a graph, hypergraph, or digraph $Q$, a set of edges $F\sseq E(\bbq)$ and any set of vertices $W\sseq V(Q)$, the \emph{flip of $F$ in $W$} is the set
    \begin{align*}
        F^{\text{f}(W)}\defined (F \setminus \set{u^{(i)}v^{(i)}\st u,v\in W, i\in\set{0,1}}) \union \set{u^{(1-i)}v^{(1-i)}\st u^{(i)}v^{(i)}\in F, u,v\in W}
    \end{align*}
    In other words, edges of $F$ in the top and bottom bunk are flipped on the subgraph induced by $W$.
    
    Then, a subset $A\sseq\cE(\bbq)$ is said to be \emph{symmetric} if for every $F\in A$ and every $W\sseq V(Q)$, we have $F^{\text{f}(W)}\in A$. 
    Moreover, we will say that $\cE(\bbq)$ is symmetric if $A=\cE(\bbq)$ is symmetric as a subset of itself.
\end{definition}

In particular, all bond percolation models considered here will be symmetric.
However, the site percolation models that we work with are in fact not symmetric by the above definition, and so we need another, site percolation-specific definition of symmetry.

\begin{definition}[Symmetric site percolation models]
    Given a graph $G$ and site percolation model $\cE(\bbg)$, take some $F\in\cE(\bbg)$ and let $U\sseq V(\bbg)$ be the corresponding set of open vertices.
    Then for any set $W\sseq V(G)$, the \emph{site-flip of $F$ in $W$} is the set $F^{\text{sf}(W)}$ of edges with both ends in the set
    \begin{align*}
        U^{\text{f}(W)}\defined\set{u^{(1-i)}\st u^{(i)}\in U, u\in W}\union \set{u^{(i)}\st u^{(i)}\in U, u\notin W}.
    \end{align*}
    In other words, open vertices in the top and bottom bunk are flipped on $W$.

    Then a subset $A\sseq \cE(\bbq)$ is said to be \emph{site-symmetric} if for every $F\in A$, we have $F^{\text{sf}(W)}\in \cE(\bbq)$.
    We will likewise say that $\cE(\bbq)$ is itself site-symmetric if it is closed under site-flips.
\end{definition}

In particular, all of the site percolation models we will work with are site-symmetric.
Next, we give some precise terminology for switching between paths in the bunkbed structure $\bbq$ and the original structure $Q$.

\begin{definition}[Shadows]
    Given a path $P$ in the bunkbed structure $\bbq$, the \emph{shadow} of $P$ is defined to be the path $\partial P$ in $Q$ obtained by replacing vertices $v^{(i)}\in V(\bbq)$ with $v\in V(Q)$, dropping any repeated vertices which arise from vertical edges.
    For instance, in the example shown in \Cref{fig:example}, the shadow of the path $w^{(0)},x^{(0)},x^{(1)},u^{(1)}$ is $w,x,u$.
\end{definition}

Finally, we define a type of event which will allow us to use the defining property of symmetric models.

\begin{definition}[Posts, quasi-posts, and wall events]
    A vertex $x\in V(G)$ will be called a \emph{post} in some $F\in\cE(G)$ if the vertical edge $x^{(0)}x^{(1)}$ is in $F$. 
    Extending this, we call a vertex $x$ a \emph{quasi-post} in some $F\in\cE(\bbq)$ if either $x$ is a post in $F$, or one can find a post $y$ and paths $P,P'$ in $Q$ from $x$ to $y$ and $y$ to $x$ respectively such that every edge $e$ in $P$ or $P'$ has both $e^{(0)}\in F$ and $e^{(1)}\in F$ (and so $x^{(0)}\biconn{F}x^{(1)}$ and $x$ behaves as if it were a post).
    Note that if $Q$ is not directed then we may take $P=P'$.
    
    An event $B\sseq\cE(\bbq)$ is then said to be a \emph{$(u,v)$-wall event} for vertices $u,v\in V(Q)$ if $B$ is either symmetric or site-symmetric, and, for every $F\in B$, every path in $F$ from $u$ to $v$ contains a quasi-post $w$ (depending on the path).
    
    For instance, in the example shown in \Cref{fig:example}, vertex $z$ is a post, and vertices $y$ and $v$ are both quasi-posts (but $w$ is not, despite the fact that $w^{(0)}\biconn{F}w^{(1)}$), and thus $\set{F}$ is a $(u,v)$-wall event.
\end{definition}

Note that a quasi-post in a site percolation model is simply a post, as $x\in V(G)$ being a quasi-post requires that both $x^{(0)}$ and $x^{(1)}$ are open.

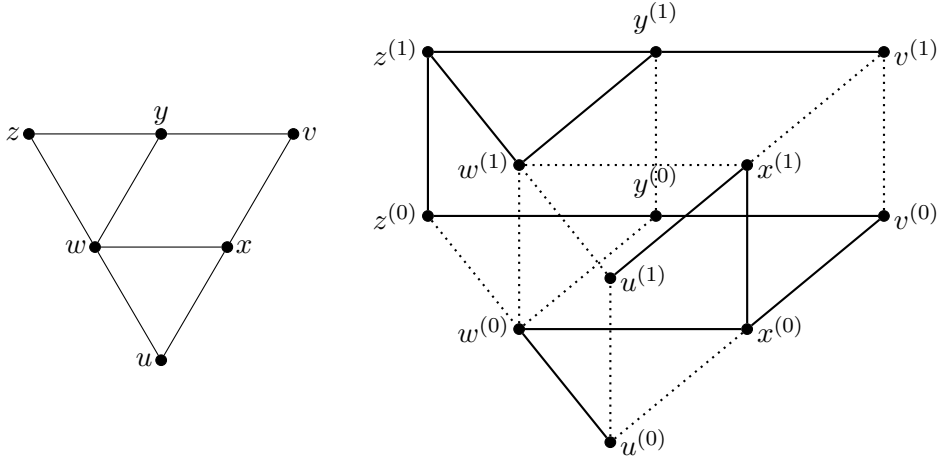
\begin{figure}[htbp]
    \centering
    \def\bunkbedheight{1.45}
\def\leftsinglewidth{0.58}
\begin{tikzpicture}[scale=1.5, every node/.style={draw, circle, fill=black, inner sep=0pt, minimum width=4pt}]

    \node[label=left:$u$] (u) at (2*\leftsinglewidth,0) {};
    \node[label=left:$w$] (w) at (\leftsinglewidth,1) {};
    \node[label=right:$x$] (x) at (3*\leftsinglewidth,1) {};
    \node[label=left:$z$] (z) at (0,2) {};
    \node[label=above:$y$] (y) at (2*\leftsinglewidth,2) {};
    \node[label=right:$v$] (v) at (4*\leftsinglewidth,2) {};

    \draw (x) -- (u) -- (w) -- (x) -- (v) -- (y) -- (w) -- (z) -- (y);

    \begin{scope}[shift={(3.5,-\bunkbedheight / 2.0)}]
        \node[label=right:$u^{(0)}$] (u0) at (1.6,0) {};
        \node[label=left:$w^{(0)}$] (w0) at (0.8,1) {};
        \node[label=right:$x^{(0)}$] (x0) at (2.8,1) {};
        \node[label=left:$z^{(0)}$] (z0) at (0,2) {};
        \node[label=above:$y^{(0)}$] (y0) at (2,2) {};
        \node[label=right:$v^{(0)}$] (v0) at (4,2) {};
    
        \begin{scope}[shift={(0,\bunkbedheight)}]
            \node[label=right:$u^{(1)}$] (u1) at (1.6,0) {};
            \node[label=left:$w^{(1)}$] (w1) at (0.8,1) {};
            \node[label=right:$x^{(1)}$] (x1) at (2.8,1) {};
            \node[label=left:$z^{(1)}$] (z1) at (0,2) {};
            \node[label=above:$y^{(1)}$] (y1) at (2,2) {};
            \node[label=right:$v^{(1)}$] (v1) at (4,2) {};
        \end{scope}
    
        \draw[thick] (u0) -- (w0) -- (x0) -- (v0) -- (y0) -- (z0) -- (z1) -- (y1) -- (v1);
        \draw[thick] (u1) -- (x1) -- (x0);
        \draw[thick] (z1) -- (w1) -- (y1);
        \draw[thick, dotted] (z0) -- (w0) -- (y0) -- (y1);
        \draw[thick, dotted] (v0) -- (v1);
        \draw[thick, dotted] (w0) -- (w1) -- (x1) -- (v1);
        \draw[thick, dotted] (x0) -- (u0) -- (u1) -- (w1);
    \end{scope}
\end{tikzpicture}
    \caption{An example bunkbed graph. $G$ is shown on the left, and $\bbg$ on the right (including both solid and dotted lines). 
    The edges of $\bbg$ shown with solid lines form the set $F$.}
    \label{fig:example}
\end{figure}

We now prove a simple lemma that we will make use of on several occasions.

\begin{lemma}
\label{lem:postless-path}
    Let $Q$ be a graph, hypergraph, or digraph, and $u,v\in V(Q)$ be two vertices, and let $A$ be a $(u,v)$-wall event.
    Then
    $$\Prob(u^{(0)}\conn{\cE(\bbq)}v^{(0)}|A) = \Prob(u^{(0)}\conn{\cE(\bbq)}v^{(1)}|A).$$
\end{lemma}
\begin{proof}
    We follow an argument similar to the `mirror argument' used by Linusson \cite[Lemma 2.3]{linusson2011}, but in a slightly more general setting.
    Define the sets $A^{(0)}$ and $A^{(1)}$ by 
    $$A^{(i)}\defined\set{F\in A\st u^{(0)}\conn{F}v^{(i)}}.$$
    We need to show that $\abs{A^{(0)}} = \abs{A^{(1)}}$.
    To this end, we construct a bijection between these families of sets.

    Take some $F\in A^{(i)}$, and define $W$ to be the set of those vertices $x\in V(Q)$ such that there is some $j\in\set{0,1}$ and a path $P$ from $x^{(j)}$ to $v^{(i)}$ with edges in $F$ such that $\partial P$ is a path from $x$ to $v$ with no vertex of $\partial P$ a quasi-post, except maybe $v$.
    In the case when $Q$ is a graph, $W$ is thus the shadow of the connected component of $v$ in $\bbq\setminus\set{w^{(i)}\st w\neq v, w\text{ is a quasi-post},i\in\set{0,1}}$.
    In particular, it is immediate for any $Q$ that every vertex $y\in W$ has a path to $v$ entirely contained in $W$. 
    Furthermore, as $A$ is a $(u,v)$-wall event, we have that $u\notin W$.
    
    By the above observations, there must be some path $P$ in $F$ from $u^{(0)}$ to $v^{(i)}$ which passes from outside of $W$ to inside $W$ exactly once, with the final vertex $x$ of $\partial P\inter (V(Q)\setminus W)$ being a quasi-post.
    Define the set $X$ to be $W\union\set{x\in V(Q)\st x\text{ is a quasi-post}}$.
    Let $F'$ and $P'$ be the flips of $F$ and $P$ in $X$ respectively, adding or removing the edge $x^{(0)} x^{(1)}$ to or from $P'$ as necessary so that $P'$ is a path.
    Note that in the case that $B$ is site-symmetric, every vertex in $X$ adjacent to a vertex not in $X$ must be a post, so flips and site-flips in $X$ coincide.
    Thus $P'$ will indeed be a path, and $F'$ is in $B$, as $(u,v)$-wall events are by definition (site-)symmetric.
    
    Then $F'\in A^{(1-i)}$ as $P'$ is a path from $u^{(0)}$ to $v^{(1-i)}$.
    Note that if we apply the same argument as above to $F'$, we produce the same set $W$ as we did for $F$, as the set of quasi-posts is unchanged by applying a flip to $F$.
    Thus the map from $F$ to $F'$ is a bijection, and $\abs{A^{(0)}} = \abs{A^{(1)}}$, as required.
\end{proof}

Throughout the paper, we will refer to models as either `conditioned' or `unconditioned'.
While this is not precise terminology, it should be taken to mean the following.
In the unconditioned cases, our model will be equivalent to one in which every edge (or vertex, in the case of site percolation) is retained independently with probability 1/2.
The conditioned cases then have further restrictions that move us away from this situation.
We now put our terminology and result to use in resolving the case of site percolation.


\section{Site percolation}
\label{sec:site-percolation}

The goal of this section is to show that the bunkbed conjecture is false for site percolation.
We consider the case of site percolation before moving on to the other structures (hypergraphs and digraphs) for two purposes: for ease of notation, as we are in this section concerned only with ordinary graphs, and to introduce our counterexamples in a more structured manner, as the results of \Cref{sec:models} are built on top of the results of this section.
In the framework introduced in \Cref{sec:intro,sec:prelims}, the site percolation we consider comes in two flavours.

\begin{model}[$E_1$; unconditioned site percolation]
\label{model:site-advanced}
    For an undirected graph $G$, let the set of open vertices $W\sseq V(\bbg)$ be chosen uniformly at random among all subsets of vertices.
    We then retain precisely those edges with both endpoints in $W$.
\end{model}

In order to construct a counterexample to the bunkbed conjecture for site percolation, we first consider a conditioned form of site percolation, in which we fix which vertices are to be posts.
We also condition further, insisting that among those $v\in V(G)$ which are not posts, precisely one of the vertices $v^{(0)}$ and $v^{(1)}$ is open.

\begin{model}[$E_2^T$; conditioned site percolation]
\label{model:site-simple}
    For an undirected graph $G$ and subset $T\sseq V(G)$, a set of open vertices is selected uniformly at random from all sets of the following form.
    If $v\in T$, both $v^{(0)}$ and $v^{(1)}$ are open, and for $v\notin T$ precisely one of $v^{(0)}$ and $v^{(1)}$ is open, and the other is closed.
\end{model}

We now construct our counterexamples for site percolation, beginning with the case of conditioned site percolation defined as \Cref{model:site-simple}.


\subsection{Conditioned site percolation}
\label{subsec:conditioned-site}

We start by giving the construction of the counterexample to the bunkbed conjecture for model $E_2^{T_2}$ as defined in \Cref{model:site-simple}.
Indeed, this graph $G_2$ is shown in \Cref{fig:counterexample}.

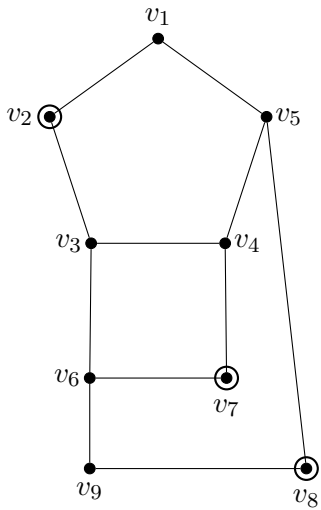
\begin{figure}[htbp]
    \centering
    \begin{tikzpicture}[scale=1.5, every node/.style={draw, circle, fill=black, inner sep=0pt, minimum width=4pt}]

    \node[label=above:$v_1$] (A) at (90:1) {};
    \node[label={[xshift=-1mm]left:$v_2$}] (B) at (162:1) {};
    \node[label=left:$v_3$] (C) at (234:1) {};
    \node[label=right:$v_4$] (D) at (306:1) {};
    \node[label=right:$v_5$] (E) at (18:1) {};
    
    \node[label=left:$v_6$] (F) at (-0.6, -2) {};
    \node[label={[yshift=-1mm]below:$v_7$}] (G) at (0.6, -2) {};

    \node[label={[yshift=-1mm]below:$v_8$}] (H) at (1.3, -2.8) {};
    \node[label=below:$v_9$] (I) at (-0.6, -2.8) {};
    
    \draw (A) -- (B) -- (C) -- (D) -- (E) -- (A);
    
    \draw (C) -- (F) -- (G) -- (D);

    \draw (E) -- (H) -- (I) -- (F);

    \draw[thick] (B) circle (0.1);
    \draw[thick] (G) circle (0.1);
    \draw[thick] (H) circle (0.1);
\end{tikzpicture}
    \caption{The basic counterexample, $G_2$. Vertices $v_2,v_7,v_8$ form the set $T_2$ conditioned to have vertical edges, and are shown highlighted with extra circles.}
    \label{fig:counterexample}
\end{figure}

Before proceeding to the proof that the graph $G_2$ in \Cref{fig:counterexample} is indeed a counterexample to the bunkbed site percolation conjecture, we first make a brief remark about its discovery.

\begin{remark}
\label{remark:gadget}
    The vertices $v_8$ and $v_9$ in $G_2$ serve as a gadget, forcing any vertex adjacent to $v_8$ (in this case, only $v_5$) to be unreachable.
    More precisely, note that $v_9$ has a unique neighbour, $v_6$, and consider removing the vertices $v_8$ and $v_9$ to form a graph $G'$, and making $v_6$ the target vertex.
    We claim that the site percolation bunkbed conjecture holding in $G$ is equivalent to it holding in $G'$ with the added constraint that no $G$-neighbour of $v_8$ is reachable from $v_1^{(0)}$.
    Indeed, if there is a path from $v_1^{(0)}$ to a neighbour of $v_8^{(0)}$ or $v_8^{(1)}$, then as $v_8$ is a post, then there is a path from $v_1^{(0)}$ to $v_8^{(0)}$ and to $v_8^{(1)}$, and thus to $v_9$, in whichever bunk it occupies.
    By symmetry of whether $v_9$ is in the upper or lower bunk, these situations may be disregarded, and so we are left to consider the case in which no neighbour of $v_8$ can be reached from $v_1^{(0)}$.
\end{remark}

The following result demonstrates that $\cone$ is false.

\begin{proposition}
\label{prop:counterexample-1}
    If $G_2$ is the graph shown in \Cref{fig:counterexample} with $T_2\defined\set{v_2,v_7,v_8}$, then we have
    $$\Prob(v_1^{(0)}\connone v_9^{(0)})=12/64<\Prob(v_1^{(0)}\connone v_9^{(1)})=13/64$$
\end{proposition}

\begin{proof}
    We prove \Cref{prop:counterexample-1} by exhaustive enumeration of cases. 
    A diagram of $G_2$ is repeated below for convenience.
    Firstly, note that we may assume by symmetry that vertex $v_1$ is present in the lower bunk.
    Furthermore, as $v_9$ is present in only one of the two bunks, we may condition on the rest of the graph, and then see in which of the upper or lower bunk $v_9$ would be reachable.
    Finally, by \Cref{remark:gadget}, if $v_5$ is in the same position (upper or lower) as $v_1$, then there is necessarily a path in $\bbg_1$ with shadow $v_1\to v_5\to v_8\to v_9$, whether $v_9$ is in the upper or lower bunk.
    Thus it suffices to consider the cases in which $v_5$ is in the upper bunk.
    
    There are three vertices left to consider: $v_3$, $v_4$, and $v_6$.
    We enumerate the eight possible configurations of these vertices, and in which states $v_9$ can be reached, in \Cref{table:enumerate}.

    As can be seen, the number of configurations in which $v_9$ is reached in the upper bunk exceeds the number in which $v_9$ is reached in the lower bunk, thus proving \Cref{prop:counterexample-1}.
    
    \begin{figure}[htbp]
        \begin{minipage}{0.65\textwidth}
            \centering
            \begin{tabular}{ccc|c}
            \toprule
            $v_3$ & $v_4$ & $v_6$ & $v_9$ \\
            \midrule
            - & - & - & $\set{-}$ \\
            - & - & + & $\set{+}$ \\
            - & + & - & $\set{-,+}$ \\
            - & + & + & $\emptyset$ \\
            + & - & - & $\emptyset$ \\
            + & - & + & $\set{+}$ \\
            + & + & - & $\set{-,+}$ \\
            + & + & + & $\set{-,+}$ \\
            \bottomrule
            \end{tabular}
            \caption{The positions in which $v_9$ can be reached in $G_2$. We assume $v_1$ is present in the lower bunk and $v_5$ in the upper.
            `-' denotes that the relevant vertex is present in the lower bunk, and `+' the upper.
            The fourth column shows in which bunks $v_9$ could be present and reachable from $v_1$.}
            \label{table:enumerate}
        \end{minipage}\hfill
        \begin{minipage}{0.35\textwidth}
            \centering
            \begin{tikzpicture}[scale=1.5, every node/.style={draw, circle, fill=black, inner sep=0pt, minimum width=4pt}]

    \node[label=above:$v_1$] (A) at (90:1) {};
    \node[label={[xshift=-1mm]left:$v_2$}] (B) at (162:1) {};
    \node[label=left:$v_3$] (C) at (234:1) {};
    \node[label=right:$v_4$] (D) at (306:1) {};
    \node[label=right:$v_5$] (E) at (18:1) {};
    
    \node[label=left:$v_6$] (F) at (-0.6, -2) {};
    \node[label={[yshift=-1mm]below:$v_7$}] (G) at (0.6, -2) {};

    \node[label={[yshift=-1mm]below:$v_8$}] (H) at (1.3, -2.8) {};
    \node[label=below:$v_9$] (I) at (-0.6, -2.8) {};
    
    \draw (A) -- (B) -- (C) -- (D) -- (E) -- (A);
    
    \draw (C) -- (F) -- (G) -- (D);

    \draw (E) -- (H) -- (I) -- (F);

    \draw[thick] (B) circle (0.1);
    \draw[thick] (G) circle (0.1);
    \draw[thick] (H) circle (0.1);
\end{tikzpicture}
            \caption{The graph $G_2$.}
        \end{minipage}
    \end{figure}    
\end{proof}

We will use this graph to produce counterexamples to the bunkbed conjecture for the other models.
First, we remove the conditioning.


\subsection{Unconditioned site percolation}
\label{subsec:unconditioned-site}

We now use graph $G_2$ to build a counterexample for \Cref{model:site-advanced}.
We construct a graph $G_1$ by taking $G_2$ and blowing up each of the three vertices in $T_2$ to an empty graph on $k$ vertices, for some $k$ to be decided later.
We now prove that the statement $\ctwo$ is indeed false.

To prove this result, we shall go via an intermediate model, which we will denote by $E_3^T$.

\begin{model}[$E_3^T$; semi-conditioned site percolation]
    For an undirected graph $G$ and subset $T\sseq V(G)$, define $T_\pm\defined \set{x^{(0)},x^{(1)}\st x\in T}$.
    Let $W\sseq V(\bbg)\setminus T_\pm$ be chosen uniformly at random amongst all such subsets.
    Then the vertices $W\union T_\pm$ are set to be open, and all other vertices are closed.
\end{model}

Note that in model $E_3^T$, vertices $x\notin T$ may or may not be posts, depending on whether or not both $x^{(0)}$ and $x^{(1)}$ are open.

Let $U_{v_2},U_{v_7}$, and $U_{v_8}$ be the sets, each of $k$ vertices, in $G_1$ produced by the blow-ups of vertices $v_2,v_7,$ and $v_8$ respectively.
Let $A$ be the event that for each $v\in T_2$, there is some $x\in U_v$ such that $x^{(0)}$ and $x^{(1)}$ are both open, so edge $x^{(0)}x^{(1)}$ is retained.

Note first that, conditional on $A$, for each $v\in T_2$, the set $U_v$ behaves exactly the same way as the post $v$ in model $E_2^T(G_2)$.
Thus for any $x,y\in V(G_2)\setminus T_2\sseq V(G_1)$, 
$$\Prob(x^{(i)}\conntwo y^{(j)}|A) = \Prob(x^{(i)}\biconn{E_3^{T_2}(G_2)} y^{(j)}).$$

Next, let $B$ be the event that some vertex $x$ in $V(G_2)\setminus T_2\sseq V(G_1)$ has either both $x^{(0)}$ and $x^{(1)}$ open (so vertical edge $x^{(0)}x^{(1)}$ is retained), or neither are open.
Then, noting that there is a unique path in $G_2$ from $v_1$ to $v_9$ avoiding $T_2$, we see that $A\inter B$ is a $(v_1,v_9)$-wall event for model $E_3^{T_2}$: if $x^{(0)}$ and $x^{(1)}$ are both open, then $x$ is a post.
If neither are open, then any path from $v_1$ to $v_9$ must avoid $x$, and thus hit $T_2$.
Therefore, by \Cref{lem:postless-path},
$$\Prob(v_1^{(0)}\conntwo v_9^{(1)}|A,B) = \Prob(v_1^{(0)}\conntwo v_9^{(0)}|A,B).$$
Furthermore, conditioning model $E_3^{T_2}$ on $B^C$ tells us that all vertices outside of $T_2$ appear in either the upper or lower bunks, but not both.
Thus
$$\Prob(v_1^{(i)}\conntwo v_9^{(j)}|A,B^C) = \Prob(v_1^{(i)}\biconn{E_3^{T_2}(G_2)} v_9^{(j)}|B^C) = \Prob(v_1^{(i)}\connone v_9^{(j)})$$
for all $i,j\in\set{0,1}$.
We therefore find that
\begin{align*}
    \Prob(v_1^{(0)}&\conntwo v_9^{(1)}) - \Prob(v_1^{(0)}\conntwo v_9^{(0)}) \\
    &= \Prob(A)\Prob(B^C)\bigl[ \Prob(v_1^{(0)}\conntwo v_9^{(1)}|A,B^C) - \Prob(v_1^{(0)}\conntwo v_9^{(0)}|A,B^C) \bigr] \\
    &\qquad + \Prob(A^C)\bigl[ \Prob(v_1^{(0)}\conntwo v_9^{(1)}|A^C) - \Prob(v_1^{(0)}\conntwo v_9^{(0)}|A^C) \bigr] \\
    &\geq 2^{-6}(1 - \Prob(A^C))\Prob(B^C) - \Prob(A^C) \\
    &= 2^{-6}\Prob(B^C) - (1 + 2^{-6}\Prob(B^C))\Prob(A^C).
\end{align*}

It thus suffices to prove that
\begin{align}
\label{eq:target}
    \Prob(A^C) < \frac{2^{-6}\Prob(B^C)}{1 + 2^{-6}\Prob(B^C)}.
\end{align}

It is easily computed that $\Prob(B^C)=2^{-6}$ and $\Prob(A)=(1-(3/4)^k)^3$, whence $\Prob(A^C)\leq 3(3/4)^k$.
In particular, inequality \eqref{eq:target} holds when $k\geq 33$.
Thus we may take $k=33$ to produce a graph $G_1$ on $6+3k=105$ vertices for which $\ctwo$ is false.


\subsection{Positive results}
\label{subsec:positive}

We now prove \Cref{thm:simple}, which states that $\conj{E_1,G}$, i.e. the bunkbed conjecture for site percolation, is true when $G$ is either a path, a cycle, or a wheel.

\begin{proof}[Proof of \Cref{thm:simple}]
    We will in fact prove the result $\conj{E_2^T,G}$ holds for any subset $T$ of $G$, and will deal with the cases of paths, cycles, and wheels in turn.
    Note that conditioning on any vertex of $G$ being present in neither the upper nor the lower bunk either trivialises the result or reduces us to an earlier case or smaller $G$.
    Thus proving the result for $E_2^T$ for arbitrary $T$ suffices to conclude the desired result for $E_1$.
    Let $B^-$ and $B^+$ be the events that $u^{(0)}\conn{} v^{(0)}$ and $u^{(0)}\conn{} v^{(1)}$ respectively.
    Note that by \Cref{lem:postless-path}, if every path from $u$ to $v$ includes an element of $T$, then we have that $\Prob(B^-)-\Prob(B^+) = 0$, which suffices to conclude our result.
    We will now deal with the three cases in turn.
    
    Firstly, if $G$ is a path, then we may restrict our attention to the section of path between $u$ and $v$. 
    If $T$ is non-empty, then every path from $u$ to $v$ passes through a post, and we are done.
    If $T$ is empty, then $\Prob(B^+)=0$, and we are again done.

    Next, assume that $G$ is a cycle; there are two paths from $u$ to $v$.
    We may assume that $T$ is non-empty, and if there is a post on each path from $u$ to $v$, then \Cref{lem:postless-path} tells us that $\Prob(B^-)-\Prob(B^+) = 0$, and we are done.
    Thus assume that there are posts on precisely one of the two paths; let $t\in T$ be one such post, and let $P$ be the vertices of the path from $t$ to $v$ not via $u$.
    Consider the function $f$ on $E_2^T(G)$ which replaces the set of open vertices with its site-flip on $P$.
    Then $f$ is a bijection, and $u^{(0)}\conn{F} v^{(1)}$ implies that $u^{(0)}\conn{f(F)} v^{(0)}$; this is sufficient to conclude.

    Finally, assume that $G$ is a wheel, and let $c$ be the centre vertex.
    If $c\in T$ then the result is trivial, so we may assume that $c$ is open in precisely one bunk.
    Let $A^-$ and $A^+$ be the events that $c$ is open in the lower and upper bunk respectively.
    Let $r$ be the probability that there is a path, not via $c$, from some vertex of $T$ to $v$ in which every vertex is open in the lower bunk (or equivalently every vertex is open in the upper bunk).
    Furthermore, let $s$ be the probability that $u^{(0)}$ connects to $c^{(1)}$ when $c$ is only present in the upper bunk; equivalently, $s$ is the probability that $u^{(0)}$ connects to some element of $T$.
    Noting that the above two events are independent, we discover the following.
    \begin{align*}
        &\Prob(B^-|A^-) = 1/2, & \Prob(B^+|A^-)&=r, \\
        &\Prob(B^-|A^+) > sr, & \Prob(B^+|A^+)&=s/2.
    \end{align*}
    The inequality for $\Prob(B^-|A^+)$ comes from the possibility that there could be a path from $u^{(0)}$ to $v^{(0)}$ not via $c$.
    We then have
    \begin{align*}
        \Prob(B^-|A^-) - \Prob(B^+|A^-) + \Prob(B^-|A^+) - \Prob(B^+|A^+) > (1-s)(1/2 - r)\geq 0,
    \end{align*}
    from which the result follows.
\end{proof}

While the above proof was not complicated, it seems to be far from simple to find a more general family of graphs for which the bunkbed conjecture for site percolation can be proved.
Some natural candidates are graphs with small diameter or radius.
Further possible extensions are discussed in \Cref{sec:conclusion}.


\section{Models}
\label{sec:models}

We consider a total of four further models in addition to those discussed in \Cref{sec:site-percolation}, two for each of the cases of hypergraphs and directed graphs.
Following the structure of the proofs in \Cref{sec:site-percolation}, in each case we will consider first a conditioned model, and then an unconditioned version.
In particular, in all cases, the conditioned model on a structure $Q$ (hypergraph or digraph) will have precisely one of the (hyper)edges $e^{(0)},e^{(1)}$ retained for each edge $e\in E(Q)$.

We will first consider hypergraphs, and now define the models we will consider.

\begin{model}[$E_4^T$; conditioned hypergraphs]
\label{model:hypergraph-simple}
    For a hypergraph $\cH$ and subset $T\sseq V(\cH)$, let $E_4^T(\bbh)$ be the set of those $F\sseq E(\bbh)$ containing precisely one of $e^{(0)}$ and $e^{(1)}$ for each $e\in E(\cH)$, and precisely those vertical edges $v^{(0)}v^{(1)}$ for which $v\in T$.
\end{model}

It was conjectured by Linusson \cite{linusson2011} that a version of the bunkbed conjecture should also hold on hypergraphs.
In our notation, the conjecture is as follows.

\begin{conjecture}[\cite{linusson2011} Conjecture 2.7]
    For any hypergraph $\cH$ and set $T\sseq V(\cH)$ of posts, $\conj{E_4^T,\cH}$ holds. 
    That is, for any vertices $u,v\in V(\cH)$, we have
    $$\Prob(u^{(0)}\biconn{E_4^T(\cH)}v^{(0)})\geq \Prob(u^{(0)}\biconn{E_4^T(\cH)}v^{(1)}).$$
\end{conjecture}

Note that \Cref{thm:main} disproves this conjecture.
Similarly to in the site percolation case, we will use our results concerning \Cref{model:hypergraph-simple} to prove a result about an unconditioned model.

\begin{model}[$E_5$; unconditioned hypergraphs]
\label{model:hypergraph-advanced}
    For a hypergraph $\cH$, let $F\sseq E(\bbh)$ be chosen uniformly at random amongst all such subsets.
    Keep precisely these edges, and discard the others.
\end{model}

Finally, we consider digraphs, and, following the same pattern as before, have two models.

\begin{model}[$E_6^{T,F}$; conditioned digraphs]
\label{model:digraph-simple}
    For a digraph $D$, subset $T\sseq V(D)$ of vertices, and subset $F\sseq E(D)$ of (directed) edges, let $E_6^{T,F}(\bbd)$ be the family of those subsets of $E(\bbd)$ containing both $e^{(0)}$ and $e^{(1)}$ for all $e\in F$, exactly one of $e^{(0)}$ and $e^{(1)}$ for all $e\in E(D)\setminus F$, and precisely those vertical edges $v^{(0)}v^{(1)}$ for which $v\in T$.
\end{model}

It was conjectured by Leander \cite{leander2009thesis} that the bunkbed conjecture holds for acyclic directed graphs (in the same sense as discussed here), and it was furthermore stated that this should broaden to hold for any directed graph.
The latter is disproved by \Cref{thm:main}, but the example we construct has many cycles, and so it remains open whether the conjecture holds for acyclic directed graphs.

We can now remove the conditioning from this model, but at the cost of working with multigraphs.

\begin{model}[$E_7$; unconditioned digraphs]
\label{model:digraph-advanced}
    For a directed multigraph $D$, let $F\sseq E(\bbd)$ be chosen uniformly at random amongst all such subsets.
    Keep precisely these edges, and discard the others.
\end{model}

As stated in $\Cref{thm:main}$, all of $\conj{E_4^T,\cH}$, $\conj{E_5,\cH}$, $\conj{E_6^{T,F},D}$, and $\conj{E_7,D}$ are false for suitable hypergraphs $\cH$, digraphs $D$, subsets of vertices $T$ and subsets of edges $F$.
In the next section, we give constructions of such graphs.


\section{Construction of counterexamples}
\label{sec:constructions}

We now construct the counterexamples for our models of percolation on hypergraphs and digraphs.

\subsection{Hypergraphs}
\label{subsec:hyper}

We first construct a hypergraph $\cH_4$ and subset $T_4\sseq V(\cH_4)$ of its vertices which contradicts the bunkbed conjecture for model $E_4^{T_4}$.
In fact, a small modification of the hypergraph dual of $G_2$ suffices, and the resulting hypergraph $\cH_4$ is shown in \Cref{fig:counterexample-hypergraph}.

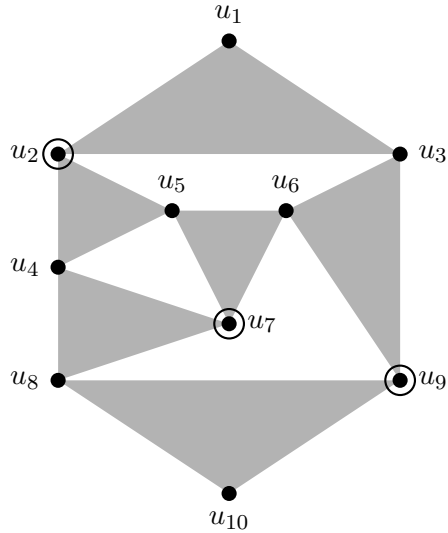
\begin{figure}[htbp]
    \centering
    \usetikzlibrary{shapes.geometric} 
\begin{tikzpicture}[scale=1.5,vertex/.style={circle, fill, inner sep=2pt}]
    \node[vertex, label=above:$u_1$] (v1) at (1.5, 4) {};
    \node[vertex, label=left:$u_2$] (v2) at (0, 3) {};
    \node[vertex, label=right:$u_3$] (v3) at (3, 3) {};
    \node[vertex, label=left:$u_4$] (v4) at (0, 2) {};
    \node[vertex, label=above:$u_5$] (v5) at (1, 2.5) {};
    \node[vertex, label=above:$u_6$] (v6) at (2, 2.5) {};
    \node[vertex, label=right:$u_7$] (v7) at (1.5, 1.5) {};
    \node[vertex, label=left:$u_8$] (v8) at (0, 1) {};
    \node[vertex, label=right:$u_9$] (v9) at (3, 1) {};
    \node[vertex, label=below:$u_{10}$] (v10) at (1.5, 0) {};
    
    \begin{scope}[fill opacity=0.3]
    \fill[black] (v1.center) -- (v2.center) -- (v3.center) -- cycle; 
    \fill[black] (v2.center) -- (v4.center) -- (v5.center) -- cycle; 
    \fill[black] (v3.center) -- (v6.center) -- (v9.center) -- cycle; 
    \fill[black] (v4.center) -- (v7.center) -- (v8.center) -- cycle; 
    \fill[black] (v5.center) -- (v6.center) -- (v7.center) -- cycle; 
    \fill[black] (v8.center) -- (v9.center) -- (v10.center) -- cycle; 
    \end{scope}

    \draw[thick] (v2.center) circle (0.13);
    \draw[thick] (v7.center) circle (0.13);
    \draw[thick] (v9.center) circle (0.13);
\end{tikzpicture}
    \caption{The basic hypergraph counterexample, $\cH_4$. All hyperedges have size 3, and shown by shaded grey triangles. Vertices $u_2,u_7,u_9$ are conditioned to have vertical edges, and are shown highlighted with extra circles.}
    \label{fig:counterexample-hypergraph}
\end{figure}

While it is not too tedious to manually check that $\cthree$ is false, we now deduce this from the previous results.
Indeed, the hypergraph $\cH_4$ is constructed in the following manner:
\begin{itemize}
    \item Let $\cH'$ be the hypergraph dual of $G_2$; that is, $V(\cH')=E(G_2)$, and for each $v\in V(G_2)$, we add a hyperedge $e_v$ to $E(\cH')$, connected to the vertices of $\cH'$ formed from those edges $e\in E(G_2)$ such that $v\in e$.
    \item Collapse the edges in $\cH'$ which came from posts in $G_2$ down to single points to form $\cH''$, and again condition these vertices to be posts, as shown in \Cref{fig:counterexample-hypergraph}.
    Let $T_4=\set{u_2,u_7,u_9}$ be this set of posts.
    \item Add two extra vertices to $\cH''$, labelled $u_1$ and $u_{10}$ in edges $e_{v_1}$ and $e_{v_9}$ respectively, to form $\cH_4$.
    These will be the new start and end vertices for our percolation.
\end{itemize}

The following claim then verifies that $\cthree$ is false.

\begin{claim}
    For $G_2$ and $\cH_4$ as shown in \Cref{fig:counterexample} and \Cref{fig:counterexample-hypergraph} respectively, for all $i,j\in\set{0,1}$,
    \begin{align}
    \label{eq:equiv-1-3}
        \Prob(u_1^{(i)}\biconn{\mthree}u_{10}^{(j)})=\Prob(v_1^{(i)}\biconn{\mone}v_9^{(j)}).
    \end{align}
\end{claim}

\begin{proof}
    Upon switching to the hypergraph dual $\cH'$ of $G_2$, collapsing $e_{v_2},e_{v_7}$, and $e_{v_9}$ down to points and conditioning that they be posts, and percolating on edges rather than vertices, the model $\mthree$ is almost identical to $\mone$.
    Indeed, for any $x,y\in V(G_2)\setminus T_2$, we have
    $$\Prob(x\biconn{\mone}y)=\Prob(e_x\biconn{\mthree}e_y).$$
    We then add vertices $u_1$ and $u_{10}$ to measure connections between $e_{v_1}$ and $e_{v_9}$ respectively within the confines of model of $\cthree$.

    Thus it is immediate that equality \eqref{eq:equiv-1-3} holds, and so $\cthree$ is false for $T_4=\set{u_2,u_7,u_9}$ and percolation from $u_1$ to $u_{10}$.
\end{proof}

Our next task is to generalise the construction of $\cH_4$ to produce a hypergraph $\cH_5$ for which the unconditioned statement $\cfour$ is false.

We may prove this following similar lines to the case of site percolation; we attach gadgets to vertices in $T_4$ which force those vertices to act as if they were conditioned to be posts with probability close to 1.
We then apply \Cref{lem:postless-path} to show that we are often in a situation equivalent to model $\mthree$.
We now present this argument in detail.

Construct $\cH_5$ by taking $\cH_4$ and adding $3k$ vertices (for integer $k>0$ to be determined later), $w_{t,i}$ attached to $u_t$ for each $t\in T_4$ and $i\in [k]$.
In other words, $k$ pendant 2-edges are attached to each of $u_2,u_7$, and $u_9$.
We will consider $\cH_4$ as a subhypergraph of $\cH_5$.

Note that if, for some $t\in T_4$ and $i\in[k]$, all three of the 2-edges $u_t^{(0)}w_{t,i}^{(0)}$, $w_{t,i}^{(0)}w_{t,i}^{(1)}$, and $w_{t,i}^{(1)}u_t^{(1)}$ are present, then $u_t$ is a quasi-post.
We now show that the probability of this occurring is close to 1 for $k$ large.

\begin{claim}
\label{claim:hypergraph-computation}
    Let $A$ be the event that for each $t\in T_4$, either the edge $u_t^{(0)}u_t^{(1)}$ is retained, or there is some $i\in[k]$ such that the 2-edges $u_t^{(0)}w_{t,i}^{(0)}$, $w_{t,i}^{(0)}w_{t,i}^{(1)}$, and $w_{t,i}^{(1)}u_t^{(1)}$ are all retained (and thus each $t\in T_4$ is a quasi-post), then $\Prob(A)\geq 1-\frac{3}{2}\bigl(\frac{7}{8}\bigr)^k$.
\end{claim}
\begin{proof}
    The probability that all three of the 2-edges $u_t^{(0)}w_{t,i}^{(0)}$, $w_{t,i}^{(0)}w_{t,i}^{(1)}$, and $w_{t,i}^{(1)}u_t^{(1)}$ are present is $1/8$, and this is independent of other edges in $\cH_5$.
    There is also a $1/2$ chance that the edge $u_t^{(0)}u_t^{(1)}$ is present, and thus the probability that $u_t$ is a quasi-post for a specific $t$ is at least $1 - \frac{1}{2}\bigl(\frac{7}{8}\bigr)^k$.
    The events we have counted are independent for different values of $t\in T_4$, and so
    \begin{align*}
        \Prob(A) \geq \Bigl[1-\frac{1}{2}\bigl(\frac{7}{8}\bigr)^k\Bigr]^3
        \geq 1-\frac{3}{2}\bigl(\frac{7}{8}\bigr)^k,
    \end{align*}
    as required.
\end{proof}

Now we apply \Cref{lem:postless-path} to the event $B$ that either some $e\in E(\cH_4)\sseq E(\cH_5)$ has either both or neither of $e^{(0)}$ and $e^{(1)}$ present, or that some $v\in V(\cH_4)\setminus T_4$ has the 2-edge $v^{(0)}v^{(1)}$ present.
Note that $A\inter B$ is a $(u_1,u_{10})$-wall event, as if $f^{(0)}$ and $f^{(1)}$ are both retained for some hyperedge $f\in E(\cH_4)$, then all vertices of $f$ are quasi-post, as $f$ contains a vertex from $T_4$.
We thus find
\begin{align*}
    \Prob(u_1^{(0)}\connfour u_{10}^{(0)}| A,B) = \Prob(u_1^{(0)}\connfour u_{10}^{(1)} | A,B).
\end{align*}

Furthermore, conditioning on $A\inter B^C$ tells us that vertices in $V(\cH_4)\setminus T_4$ are not posts, and edges in $E(\cH_4)$ appear in precisely one of the upper or lower bunks.
Therefore
$$\Prob(u_1^{(i)}\connfour u_{10}^{(j)} | A,B^C) = \Prob(u_1^{(i)}\connthree u_{10}^{(j)}).$$

We may now proceed exactly as in the previous section, and discover that it again suffices to prove inequality \eqref{eq:target}.
Indeed, it is easy to compute that $\Prob(B^C)=2^{-13}$, and so together with \Cref{claim:hypergraph-computation}, we find by elementary algebra that inequality \eqref{eq:target} holds for $k > \log_{8/7}(2\cdot 2^{18} + 3/2)\approx 101.663$.
Thus we may take $k=102$ to find a hypergraph $\cH_5$ on $10 + 3k = 316$ vertices for which $\cfour$ is false.

\subsection{Directed graphs}
\label{subsec:directed}

Finally, we consider directed graphs.
We produce the digraph $D_6$ from hypergraph $\cH_4$, and then construct from this the digraph $D_7$.

When building $D_6$ we condition not only on some set of vertices $T_6$ which will be posts, but on some set $F_6\sseq E(D_6)$ of edges which appear in both the upper and lower bunks, with all edges not in $F_6$ appearing in precisely one of the two bunks.
We will refer to the edges in $F_6$ as \emph{double edges}.

Our method for constructing $D_6$ from $\cH_4$ is to replace a 3-edge $f=uvw$ with the digraph described as follows, which is also shown in \Cref{fig:hyper-to-directed}.

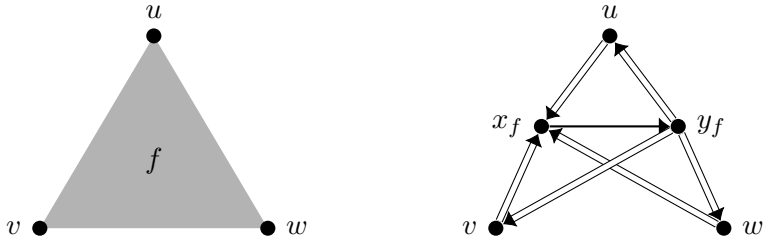
\begin{figure}[htbp]
    \usetikzlibrary{shapes.geometric} 
\usetikzlibrary{arrows.meta} 
\begin{tikzpicture}[scale=1.5,vertex/.style={circle, fill, inner sep=2pt}]
    \node[vertex, label=above:$u$] (u1) at (1, 1.7) {};
    \node[vertex, label=left:$v$] (v1) at (0, 0) {};
    \node[vertex, label=right:$w$] (w1) at (2, 0) {};
    \node (f) at (1,0.6) {$f$};
    
    \node[vertex, label=above:$u$] (u2) at (5, 1.7) {};
    \node[vertex, label=left:$v$] (v2) at (4, 0) {};
    \node[vertex, label=right:$w$] (w2) at (6, 0) {};
    \node[vertex, label=left:$x_f$] (x) at (4.4, 0.9) {};
    \node[vertex, label=right:$y_f$] (y) at (5.6, 0.9) {};
    
    \begin{scope}[fill opacity=0.3]
        \fill[black] (u1.center) -- (v1.center) -- (w1.center) -- cycle;
    \end{scope}

    \draw[thick,-{Latex[length=1.5mm, width=2.5mm]}] (x) -- (y);
    \draw[->,double,double distance=2pt,-{Latex[length=1.5mm, width=2.5mm]}] (u2) -- (x);
    \draw[->,double,double distance=2pt,-{Latex[length=1.5mm, width=2.5mm]}] (v2) -- (x);
    \draw[->,double,double distance=2pt,-{Latex[length=1.5mm, width=2.5mm]}] (w2) -- (x);
    \draw[->,double,double distance=2pt,-{Latex[length=1.5mm, width=2.5mm]}] (y) -- (u2);
    \draw[->,double,double distance=2pt,-{Latex[length=1.5mm, width=2.5mm]}] (y) -- (v2);
    \draw[->,double,double distance=2pt,-{Latex[length=1.5mm, width=2.5mm]}] (y) -- (w2);
\end{tikzpicture}
    \caption{How a hyperedge $f$ in $\cH_4$ is replaced by two vertices and several directed edges in $D_6$. All new edges except for $x_f y_f$ are double edges, and are conditioned to appear in both the upper and lower bunks.}
    \label{fig:hyper-to-directed}
\end{figure}

Add two new vertices $x_f,y_f$, along with directed edge $x_f y_f$.
Then add double edges $ux_f$, $vx_f$, $wx_f$, $y_fu$, $y_f v$, and $y_f w$ to the set $F_6$.
In this construction, the edge $x_f y_f$ acts in place of the hyperedge $f$; it can be seen that $u^{(0)}\biconn{} v^{(0)}\biconn{} w^{(0)}$ whenever $x_f y_f$ appears in the lower bunk, and $u^{(1)}\biconn{} v^{(1)}\biconn{} w^{(1)}$ whenever $x_f y_f$ appears in the upper bunk, exactly as is required for modelling a hypergraph.
The set $T_6$ is equal to $T_4$.

Thus, labelling the vertices of $D_6$ copied from $\cH_4$ with the same names $u_1,\dotsc,u_{10}$, we see that
$$\Prob(u_a^{(i)}\connfive u_b^{(j)}) = \Prob(u_a^{(i)}\connthree u_b^{(j)}),$$
for all $a,b\in [10]$ and all $i,j\in\set{0,1}$. In particular, $\cfive$ is false.\\

Thus it remains to construct directed multigraph $D_7$, and remove the conditioning of model $\mfive$; on $T_6$ and $F_6$, and the assumption that edges are present in only one bunk.
This argument follows very similar lines to those in the previous sections.

Construct $D_7$ by attaching to each post vertex $u_i\in T_6$ $k$ vertices $w_{i,1},\dotsc w_{i,k}$ (for $k$ to be determined later), along with directed edges $u_i w_{i,j}$ and $w_{i,j} u_i$ for each $i$ and $j$.
Replace each double edge $xy$ in $F_6$ with a set of $k$ parallel directed edges $e_{x,y,i}$ from $x$ to $y$ for $i\in [k]$.
Let $A$ be the event that all of the following occur:
\begin{itemize}
    \item For each $u_i\in T_6$, either the bidirectional vertical edge $u_i^{(0)} u_i^{(1)}$ is present, or for some $j\in [k]$, all four of the directed edges $u_i^{(0)} w_{i,j}^{(0)}$, $w_{i,j}^{(0)} u_i^{(0)}$, $u_i^{(1)} w_{i,j}^{(1)}$, and $w_{i,j}^{(1)} u_i^{(1)}$, and bidirectional vertical  edge $w_{i,j}^{(0)} w_{i,j}^{(1)}$ are present. In either case, $u_i$ is a quasi-post.
    \item For each $xy\in F_6$, there is some $i\in [k]$ such that edge $e_{x,y,i}^{(0)}$ is present, and some $j\in[k]$ such that edge $e_{x,y,j}^{(1)}$ is present.
\end{itemize}

Let $B$ be the event that any vertex in $V(D_6)\setminus T_6$ is a post, or that any edge in $E(D_6)\setminus F_6$ is present in both the upper and lower bunks.
Note that $A\inter B$ is a $(u_1,u_{10})$-wall event, and note further that $E_7$ conditioned on $A$ and $B^C$ is equivalent to $\mfive$.
Thus, as in previous cases,
$$\Prob(u_a^{(i)}\connsix u_b^{(j)}|A,B^C) = \Prob(u_a^{(i)}\connfive u_b^{(j)}) = 2^{-6}.$$
Following identical arguments to those before, it again suffices to prove that inequality \eqref{eq:target} holds, and all it is left to do is compute $\Prob(A^C)$ and $\Prob(B^C)$.

There are 19 vertices in $V(D_6)\setminus T_6$, and 6 edges in $E(D_6)\setminus F_6$, and so $\Prob(B^C)=2^{-25}$. Further elementary computations give us that
$$\Prob(A) = \Bigl(1-\frac{1}{2}\bigl(\frac{31}{32}\bigr)^k\Bigr)^3 \bigl(1-2^{-k}\bigr)^{72}.$$
It can then be computed that inequality \eqref{eq:target} holds for those $k\geq 690$, and so we take $k=690$.
We have thus found our directed multigraph $D_7$ on $22+3k=2092$ vertices with $6+36k = 24846$ edges for which $\csix$ is false, completing the proof of \Cref{thm:main}.


\section{Conclusion}
\label{sec:conclusion}

We have shown that the bunkbed conjecture does not hold when generalised to site percolation or hypergraphs.
The situation with directed graphs is a little more complex, as we have only dealt with the case of multigraphs with vertical edges which are bidirectional.
In the interest of considering fully directed graphs, we define the following model:

\begin{model}[$E_8$; fully directed graphs] 
    For a simple digraph $D$, construct the bunkbed graph $\bbd$ by adding both directed edges $v^{(0)}v^{(1)}$ and $v^{(1)}v^{(0)}$ for all $v\in V(D)$. 
    A subset of edges $F\sseq E(\bbd)$ is selected uniformly at random, and precisely the edges in $F$ are retained.
\end{model}

We conjecture that it is not necessary that our example in \Cref{subsec:directed} has multiple edges and undirected vertical edges.

\begin{conjecture}
    There is a simple digraph $D$ for which $\conj{E_8,D}$ is false.
\end{conjecture}

We furthermore repeat the following question of Leander \cite{leander2009thesis} in our context. 

\begin{question}
\label{question:acyclic}
    Do $\conj{E_6^{T,F},D}$ and $\conj{E_7,D}$ hold for all acyclic digraphs $D$?
\end{question}

It seems unclear to the author whether the answer to \Cref{question:acyclic} should be positive or negative.
Indeed, this problem is far more amenable to inductive techniques than the case of graphs with cycles, but actually performing this induction seems far from trivial.

However, one case that does seem approachable, though was found by the author to be more difficult than expected, is expressed by the following conjecture.

\begin{conjecture}
    $\conj{E_7,T_n}$ holds for $T_n$. the transitive tournament on $n$ vertices.
\end{conjecture}

We found some positive results in \Cref{subsec:positive} concerning the case of site percolation for some very simple classes of graphs.
It would be interesting to consider more complex graphs, with a view towards ascertaining whether or not the bunkbed conjecture failing for site percolation is in general the exception or the rule.
Indeed, one may ask the following (perhaps optimistic) question.

\begin{question}
    What is the probability that the assertion in the bunkbed conjecture for site percolation $\conj{E_1,G}$ holds for $G\sim G(n,p)$ the Erd\H{o}s-R\'{e}nyi random graph for $n$ large?
    To what extent does this probability depend on the behaviour of $p$ as a function of $n$?
\end{question}

There are of course many other classes of graph one could consider; graphs of diameter $2$ (or simply of radius 1), and highly symmetric graphs (e.g. those with transitive automorphism group) would seem to be natural candidates for investigation.

Returning now to the original bunkbed conjecture $\conj{E_0,G}$, while our results do not directly say anything about this problem, they do cast doubt on whether the conjecture is just as likely to be true as the literature often claims.
Indeed, much of the same intuition which tells us that $\conj{E_0,G}$ should hold also goes through in the cases of the other models considered here, which -- as we have shown -- are false.
Any proof of the bunkbed conjecture would thus need to make use of the fact that it is working with undirected, 2-uniform graphs.

As discussed in the introduction, it has been shown \cite{hollom2024new,hutchcroft2023bunkbed} that the bunkbed conjecture holds for any fixed graph $G$ in the $p\uparrow 1$ limit.
However, it is not hard to show that our methods can produce counterexamples in the setting of models $E_1$ and $E_5$ even when the probability $p$ tends to $1$.
Thus these constructions cannot directly be transformed into counterexamples to the original bunkbed conjecture, and some new idea would be needed here.

We close by stating the bunkbed conjecture in the form of the following question.

\begin{question}
\label{question:bunkbed}
    Does $\czero$ hold for every graph $G$ and set $T\sseq V(G)$?
\end{question}

In light of the results presented here, the author is far from convinced that the answer to \Cref{question:bunkbed} is positive.


\section{Acknowledgement}
The author would like to thank B\'{e}la Bollob\'{a}s for his thorough reading of the manuscript and many valuable comments.


\bibliographystyle{abbrvnat}  
\renewcommand{\bibname}{Bibliography}
\bibliography{main}


\end{document}